\documentclass[reqno]{amsart}
\usepackage{latexsym,amsfonts,amssymb,amsmath,euscript}
\usepackage[latin1]{inputenc} 
\usepackage{graphicx} 
\usepackage{color,fancybox
}

\newcommand{\R} {\mathbb{R}}
\newcommand{\N} {\mathbb{N}}

\newcommand{\eps}{\epsilon}

\def\eto{\buildrel \epsilon\to 0\over\longrightarrow }

\def\etow{\buildrel \epsilon\to 0\over\rightharpoonup}

\begin{document}
\title[Highly oscillatory thin domains]{The Neumann problem in thin domains with very highly oscillatory boundaries}
\thanks{ Math Subject Classification (2010): 35B27, 74K10}
\author[J. M. Arrieta]{Jos\'{e} M. Arrieta$^{*, \dag}$}
\thanks{$^*$ Corresponding author:  Jos\'e M. Arrieta,  Departamento de Matem\'atica Aplicada, Facultad de Matem\'aticas, 
Universidad Complutense de Madrid, 28040  Madrid, Spain. e-mail: arrieta@mat.ucm.es}
\thanks{$^\dag$ Partially
supported by:  MTM2009-07540  (MICINN) and MTM2012-31298 (MINECO), Spain; PHB2006-003 PC from MICINN;
     and  GR58/08 and GR35/10-A, Grupo 920894 BSCH-UCM, Spain.}
\address[Jos\'e M. Arrieta]{Departamento de Matem\'atica Aplicada,
Facultad de Ma\-te\-m\'a\-ti\-cas, Universidad Complutense de
Madrid, 28040 Madrid, Spain.} \email{arrieta@mat.ucm.es}

\author[M.C.Pereira]{Marcone C. Pereira$^{*,\ddag}$}
\thanks{$^\ddag$ Partially supported by CNPq 305210/2008-4 and 302847/2011-1, FAPESP 2008/53094-4, 2010/18790-0 and 2011/08929-3.}
\address[Marcone C. Pereira]{Escola de Artes, Ci\^encias e Humanidades \\ Universidade de S\~ao
Paulo, S\~ao Paulo, SP, Brazil} \email{marcone@usp.br}



\date{}

\begin{abstract}
In this paper we analyze the behavior of solutions of the Neumann problem posed in a thin domain of the type 
$R^\epsilon = \{ (x_1,x_2) \in \R^2 \; | \;  x_1 \in (0,1), \, - \, \eps \, b(x_1) < x_2 < \epsilon \, G(x_1, x_1/\eps^\alpha) \}$ 
with $\alpha>1$ and $\epsilon > 0$, defined by smooth functions $b(x)$ and $G(x,y)$, where the function $G$ is supposed to be 
$l(x)$-periodic in the second variable $y$. 
The condition $\alpha > 1$ implies that the upper boundary of this thin domain presents a very high oscillatory behavior.
Indeed, we have that the order of its oscillations is larger  than the order of the amplitude and height of $R^\epsilon$ 
given by the small parameter $\epsilon$.
We also consider more general and complicated geometries for thin domains 
which are not given as the graph of certain smooth functions, but rather more comb-like domains.
\end{abstract}

\maketitle
\numberwithin{equation}{section}
\newtheorem{theorem}{Theorem}[section]
\newtheorem{lemma}[theorem]{Lemma}
\newtheorem{corollary}[theorem]{Corollary}
\newtheorem{proposition}[theorem]{Proposition}
\newtheorem{remark}[theorem]{Remark}
\allowdisplaybreaks

\section{Introduction} \label{introduction}


In this paper, we analyze the behavior of the solutions of the Laplace equation with homogeneous Neumann boundary conditions
\begin{equation} \label{OPI}
\left\{
\begin{gathered}
- \Delta w^\epsilon + w^\epsilon = h^\epsilon
\quad \textrm{ in } R^\epsilon \\
\frac{\partial w^\epsilon}{\partial N^\epsilon} = 0
\quad \textrm{ on } \partial R^\epsilon
\end{gathered}
\right. 
\end{equation}
where $N^\epsilon$ is the unit outward normal to $\partial R^\epsilon$ and $h^\epsilon \in L^2(R^\epsilon)$. 
The domain $R^\epsilon$ is a  two dimensional thin domain which presents a highly oscillatory behavior at the boundary.  We will be able to consider two different types of thin domains, which will be clearly defined in Section 2.  To make the ideas clear we will refer in this introduction to the first type: assume $R^\eps$ is given as the region between two functions, that is, 
\begin{equation}\label{thin-intro}
R^\epsilon = \{ (x_1,x_2) \in \R^2 \; | \;  x_1 \in (0,1),  \;
 - \epsilon \, b(x_1) < x_2 < \epsilon \, G_\epsilon(x_1) \}\
 \end{equation}
where $b(\cdot)$ and $G_\eps(\cdot)$ are functions satisfying 
$
0<b_0<b(\cdot)<b_1
$, 
$
0 \leq G_\eps(\cdot)\leq G_1
$ 
for some fixed positive constants $b_0$, $b_1$ and $G_1$, independent of $\epsilon > 0$. 
Here, the function $b$, independent of $\epsilon$, defines the lower boundary of the thin domain, and the function $G_\epsilon$,
dependent on $\epsilon$, the upper boundary of $R^\epsilon$.
We will allow $G_\eps$ to present oscillations whose amplitude is larger than the order of compression of the thin domain. This is expressed by assuming that   
\begin{equation} \label{FGepI}
G_\eps(x)=G(x,x/\eps^\alpha),
\end{equation} 
for some positive constant $\alpha > 1$.  The function 
$G:(0,1) \times \R \to \R$ is a positive smooth function, with $y\to G(x,y)$ periodic in $y$ for fixed $x$ with period $l(x)$.

Let us observe that our assumptions includes the case where the function $G_\eps$ presents a purely periodic behavior, 
for instance, $G_\eps(x)=2+\sin(x/\eps^\alpha)$. But it also considers the case where the function 
$G_\eps$ defines a thin domain where the oscillations period, the amplitude and the profile vary with respect to $x \in (0,1)$.
Figures 1 and 2 below illustrate two examples of thin domains that we consider in this work.

\begin{figure}[htp] 
\centering \scalebox{0.55}{\includegraphics{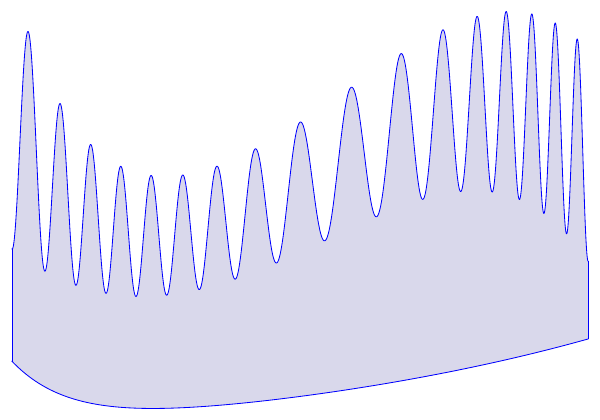}}
\caption{A thin domain with variable period, amplitude and profile.}
\end{figure}

\begin{figure}[htp]
\centering \scalebox{0.55}{\includegraphics{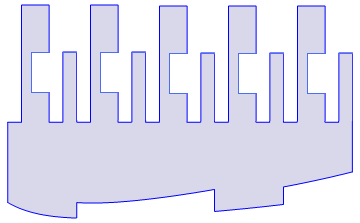}}
\caption{A comb-like thin domain.}
\label{typeII}
\end{figure}

Since the domain $R^\epsilon$ is thin, approaching the interval $(0,1)$, it is reasonable to expect that
the family of solutions will converge to a function of just one variable and that this function will satisfy
certain elliptic equation  in one dimension with some boundary conditions.

It is known that if the domain does not present oscillations, that is $G_\epsilon (x) = G(x)$, with $0<G_0\leq G(\cdot) \leq G_1$ and $b(\cdot)\equiv 0$ the 1-dimensional limiting problem is given by
\begin{equation} \label{thin-domain-hale-raugel}
\left\{
\begin{gathered}
-\frac{1}{G(x)}\Big(G(x) \, w_x\Big)_x + w = h, \hbox{ in } (0,1), \\
w_x(0) = w_x(1)=0
\end{gathered}
\right. 
\end{equation}
see for instance \cite{HR, R}.  
Also, if we consider $b \equiv 0$, $G_\eps(x)=G(x,x/\epsilon^\alpha)$ for some $0\leq\alpha<1$,
and if we assume that $G_\eps(\cdot)\to m(\cdot) \quad w-L^2(0,1)$ and $\frac{1}{G_\eps(\cdot)}\to k(\cdot) \quad w-L^2(0,1)$
(observe that $m(x)k(x)\geq 1$ a.e. and in general it is not true that $m(x)k(x)\equiv 1$), then the limit problem is 
$$
\left\{
\begin{gathered}
-\frac{1}{m(x)}\left(\frac{1}{k(x)} w_x\right)_x + w = h, \hbox{ in }(0,1)\\
w_x(0)=w_x(1)=0
\end{gathered}
\right. 
$$
see \cite{JA} for details.    
Note that this case contains the previous one, since we can recover (\ref{thin-domain-hale-raugel}) by taking $\alpha=0$.

Recently, we considered in \cite{AP2, AP} a class of oscillating thin domain that cover the case $\alpha=1$ with constant period $l$. 
Observe that this situation is very resonant since the height of the domain, the amplitude of the oscillations at the boundary and 
the period of the oscillations are of the same order $\eps$.  
The limit problem for this case is
\begin{equation} \label{GLP-cont}
\left\{
\begin{gathered}
-\frac{1}{s(x)}(r(x)w_{x})_x + w =  h(x), \quad x \in (0,1)\\
w'(0)=w'(1)=0
\end{gathered}
\right.
\end{equation} 
where 
$$
\begin{gathered}
r(x)= \int_{Y^*(x)} \Big\{ 1 - \frac{\partial X(x)}{\partial y_1}(y_1,y_2) \Big\} dy_1 dy_2, \qquad 
\end{gathered}
$$ 
$$s(x)={|Y^*(x)|}$$
and $X(x)$ is a convenient auxiliary harmonic function
defined in the representative basic cell $Y^*(x)$, which depends on $G(x,\cdot)$, $x \in (0,1)$, and 
it is given by 
$$
Y^*(x)= \{ (y_1,y_2) \in \R^2 \; | \; 0< y_1 < l, \quad 0 < y_2 < G(x,y_1) \}.
$$
The restricted case where the function $G_\eps(x)=G(x/\eps)$ for some $l$-periodic smooth function $G$ can be addressed by somehow standard techniques in homogenization theory, as developed in \cite{BLP, CP, SP}. 
We refer to \cite{ACPS} for a complete analysis of this case for a semilinear parabolic problem.

In this work, we are interested in addressing the case $\alpha>1$ in \eqref{FGepI}, 
where none of the techniques used to solve the previous ones apply. In particular,  we do not use any extension operator for
the convergence proof.
 Indeed, we will be able to show how the geometry of the boundary oscillations affect the limiting equation, see Theorem \ref{main}, Theorem \ref{main2}. See also Corollary \ref{periodic} for a very interesting interpretation of the limiting equation and to see how the geometry of the unit cell affects the limit equation in the case of periodic oscillations.

There are several works in the literature on partial differential equations dealing with the problem of  thin domains presenting oscillating boundaries. Among others, we may mention \cite{MP,MP2} who studied the asymptotic approximations of solutions to parabolic and elliptic problems in thin perforated domain with rapidly varying thickness, and  \cite{BGG,BGM} who consider nonlinear monotone problems in a multidomain with a highly oscillating boundary. 
In addiction to these, we also may cite the works \cite{AB,BZ, BFF}, in which the asymptotic description of nonlinearly elastic thin films with fast-oscillating profile was successfully obtained in a context of \emph{$\Gamma$-convergence}  \cite{Dal}.
In particular, we observe that the boundary perturbation studied in the papers \cite{AB,BFF} is related to the present one. Our goal here is allow much more complicated profiles for the oscillating thin domain obtaining the limit problem as well as its dependence with respect to the thin domain geometry. Indeed, we give an explicit relationship among the homogenized equation, the oscillation, the profile and thickness of the thin domain. Also, we are able to get strong convergence in $H^1$-norm when we compare the solutions of the limit problem and the perturbed one.

There are also some other works addressing the problem of the behavior of solutions of some partial differential equations when just the boundary of the domain presents an oscillatory behavior.   In this case,  the perturbation affects only at the boundary of the domain, that is, roughly speaking we have a fixed domain $\Omega$ and the perturbed domain is obtained by modifying part of its boundary with an oscillatory behavior. This is the case of \cite{CFP}, where the authors deal with  the Poission equation with Robin type boundary conditions in the oscillating part of the boundary.   In general, the differential equation is not affected by this perturbation but it is the boundary condition the one which is influenced by the oscillations. Some very interesting interplay among the geometry of the oscillations and the Robin boundary conditions is analyzed so that depending on this balance the limiting boundary condition may be of a different type.  Similar results are also encountered in  \cite{CLS} for the case of the Stokes operator and the Navier-Stokes equations with slip boundary conditions and \cite{ArrBru10} for the case of nonlinear elliptic equations with nonlinear boundary conditions (see also references in these paper).  
In the case of the present paper, the effect of the oscillations at the boundary is coupled with the effect of the domain being thin. The oscillations behavior affect completely at the equation, not just at the boundary conditions and the limit equation is in a lower dimensional space (1-D in this case) than the perturbed equations.

Finally, let us point out that thin structures with rough contours (thin rods,  plates or shells) or fluids filling out thin domains (lubrication) or even chemical diffusion process in the presence of grainy narrow strips (catalytic process) are very common in engineering and applied science.  The analysis of the properties of these structures and the processes taking place on them and understanding how the ``micro'' geometry of the thin structure affects the ``macro'' properties of the material is a very relevant issue in engineering  and material design.  In this respect,  being able to obtain the limiting equation of a prototype equation (like the Poisson equation) in different structures where the ``micro'' geometry is not necessarily smooth and being able to analyze how the different ``micro'' scales affects the limiting problem goes in this direction and will allow the study and understanding in more complicated situations.    
We refer to  \cite{Apl3,Apl1,Apl2} for some concrete more applied problems. 

%

This paper is organized as follows.
In Section \ref{basics} we give precise definitions of the two types of thin domains we are considering. One of them is the one described in this introduction. The other type is a ``comb-like'' thin domain, which can be visualized in Figure 2.    We also state clearly the two main results we prove, Theorem \ref{main} and Theorem \ref{main2}. 
The short Section \ref{basic-section} states a technical result which will be used later in the proof. 
In Section \ref{thin-domain-graphs} we analyze the type  of thin domains which are given as a region between two graphs as in \eqref{thin-intro}. 
In Section \ref{comb} we analyze the other type of thin domains, that we have denoted as a ``comb-like'' thin domain. 
Sections \ref{thin-domain-graphs} and \ref{comb} are dedicated to the proof of Theorems \ref{main} and \ref{main2} respectively.

We also would like to observe that although we will deal with Neumann boundary conditions,  
we may also impose different conditions in the lateral boundaries of the thin
domain $R^\eps$, while preserving the Neumann type boundary condition in the upper and lower boundary.  
Indeed, we may consider conditions of the Dirichlet type, $w^\eps=0$, or even Robin, $\frac{\partial w^\eps}{\partial N}+\beta w^\eps=0$ in the lateral
boundaries of the problem (\ref{OPI}).
The limit problem will preserve this boundary condition as a point condition.

\section{Basic facts, notation and main results}
\label{basics}

We will consider two different types of thin domains. One of them will be given as the region between the graphs of two functions and the other will consists of an autoreplicating structure with appropriate scaling rates which resembles a comb structure. We present now the main definitions, basic facts and results on both cases. 

\par\medskip\noindent{\bf Type I.  Thin domain as the region between two graphs}. 
Let us consider an one parameter family of functions $G_\eps: (0,1) \to [0,\infty)$, $\eps \in (0, \eps_0)$ for some $\eps_0>0$, 
and a function $b:(0,1) \mapsto (0,\infty)$. 
We will assume the following hypotheses on functions $b$ and $G_\eps$:
\begin{itemize}
\item[(H1)] There exist two positive constants $b_0$, $b_1$ such that $0 < b_0 \leq b(x) \leq b_1$ for all $x \in (0,1)$ and the function $b$ is piecewise $C^1$. 
\item[(H2)] The  functions $G_\eps(\cdot)$ are of 
the type $G_\eps(x)=G(x,x/\eps^\alpha)$, with $\alpha > 1$, where the function
\begin{equation}\label{def-G}
\begin{array}{rl}
G:[0,1]\times \R &\longrightarrow [0,+\infty) \\
 (x,y)&\longrightarrow G(x,y)
 \end{array}
 \end{equation}
is continuous in $x$, uniformly in the second variable $y$, (that is, for each $\eta > 0$, there exists
$\delta > 0$ such that $|G(x,y) - G(x',y)| \leq \eta$ for all $x$, $x' \in [0,1]$, $|x-x'| < \delta$, and $y \in \R$).
Moreover, we assume $G(x,y) \geq 0$ is periodic in $y$, with a period $l(x)$ that may depend on the first variable, that is, $G(x,y+l(x))=G(x,y)$. 
We also assume that $l(\cdot)$ is a continuous function  in $[0,1]$, with  $ 0 < L' \leq l(x) \leq L$ for all $x \in [0,1]$. 
\end{itemize}

We consider the highly oscillating thin domain $R^\epsilon$, which is given as the region between the graphs of 
the two functions $\eps b(\cdot)$ and $\eps G_\eps(\cdot)$, that is
$$
R^\epsilon = \{ (x_1,x_2) \in \R^2 \; | \;  x_1 \in (0,1),  \;
 - \epsilon \, b(x_1) < x_2 < \epsilon \, G_\epsilon(x_1) \}  
$$
and we investigate the behavior of the solutions of \eqref{OPI} as $\eps\to 0$.

Observe that since $\alpha > 1$, we have that the upper boundary of this thin domain presents a very high oscillatory behavior.
More precisely, the period of the oscillations is much smaller (order $\sim\eps^\alpha$) than the amplitude (order $\sim\eps$) and height of the thin domain (order $\sim\eps$). Figure 3 below gives us an example of a function $G$ in a bounded open set. 

\begin{figure}[h]
\centering \scalebox{0.5}{\includegraphics{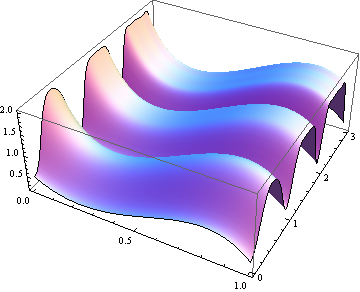}}
\caption{Graph of function $G$ in $(0,1) \times (0, 3 L)$.}
\end{figure}

To study the convergence of the solutions of (\ref{OPI}), 
we consider the equivalent linear elliptic problem
\begin{equation} \label{P}
\left\{
\begin{gathered}
 - \frac{\partial^2 u^\epsilon}{{\partial x_1}^2} - 
\frac{1}{\epsilon^2} \frac{\partial^2 u^\epsilon}{{\partial x_2}^2} + u^\epsilon = f^\epsilon
\quad \textrm{ in } \Omega^\epsilon \\
\frac{\partial u^\epsilon}{\partial x_1} \nu_1^\epsilon + \frac{1}{\epsilon^2} \frac{\partial u^\epsilon}{\partial x_2}\nu_2^\epsilon = 0
\quad \textrm{ on } \partial \Omega^\epsilon
\end{gathered}
\right.
\end{equation}
where $f^\eps\in L^2(\Omega^\eps)$  satisfies 
\begin{equation} \label{ESTF}
\| f^\epsilon \|_{L^2(\Omega^\epsilon)} \le C
\end{equation}
for some $C > 0$ independent of $\epsilon$, and now, $\nu^\eps=(\nu_1^\eps,\nu_2^\eps)$ is the outward
unit normal to $\partial\Omega^\eps$, and $\Omega^\epsilon \subset \R^2$ is a highly oscillating domain given by
\begin{equation} \label{domain}
\Omega^\epsilon = \{ (x_1,x_2) \in \R^2 \; | \;  x_1 \in (0,1), \; -b(x_1) < x_2 < G_\epsilon(x_1) \}.
\end{equation}
Note that the equivalence between (\ref{OPI}) and (\ref{P}) is easily obtained by changing the scale 
of the thin domain $R^\epsilon$ in the $y$-direction through the simple transformation $(x,y) \to (x, \epsilon y)$, 
(see \cite{JA,HR} for more details).
Thus, we have a domain which is not thin anymore but presents very wild oscillatory behavior at the top boundary, 
although the presence of a high diffusion coefficient in front of the derivative with respect the second variable decreases
the effect of the high oscillations.

It is worth also mentioning the works \cite{ABMG,BCh,CFP,DP} that analyze elliptic problems in oscillating domains related to $\Omega^\eps$ but the fact that in our case we have a very high diffusion in the $y$-direction makes the analysis and result different from theirs.  
In fact, here we are considering a situation in which a phenomenon described for a two-dimensional differential equation can be strictly approximated for one-dimensional one.

Now we are in condition to state our main result whose proof will be presented in section \ref{PofC}.

\begin{theorem} \label{main}
Assume that  $f^\epsilon \in L^2(\Omega^\epsilon)$ satisfies $\|f^\eps\|_{L^2(\Omega^\eps)}\leq C$ and the function $\hat f^\eps(x)=\int_{-b(x)}^{G_\eps(x)}f(x,y) \, dy$ satisfies that $\hat f^\eps\rightharpoonup \hat f$, w-$L^2(0,1)$.
Let $u^\epsilon$ be the unique solution of (\ref{P}) and $G_0$ be the function given by 
\begin{equation} \label{FG0T}
G_{0}(x) = \min_{y \in \R} G(x,y)\geq 0.
\end{equation}
 
Then, if  $u_0(x_1)$ is the unique 
weak solution of the Neumann problem
\begin{equation} \label{VFPDL-piecewise0}
\int_0^1 \Big\{ \Big( b(x) + G_0(x) \Big) \; u_x(x) \, \varphi_x(x) 
+ p(x) \, u(x) \, \varphi(x) \Big\} dx = \int_0^1 \hat f(x) \, \varphi \, dx, \quad \forall \varphi\in H^1(0,1)
\end{equation}
where $p(x)$ is the function defined as follows: 
$$p(x)=\frac{1}{l(x)}\int_0^{l(x)} \Big( b(x)+G(x,y) \Big) \, dy = b(x)+ \frac{1}{l(x)}\int_0^{l(x)} G(x,y) \, dy, \hbox{  for all }x\in (0,1),$$
we have 
$$
\| u^\epsilon-u_0\|_{L^2(\Omega^\eps)} \eto 0.
$$ 
Moreover, if we denote by $\Omega_0=\{ (x_1,x_2) \in \R^2 \; | \;  x_1 \in (0,1), \; -b(x_1) < x_2 < G_0(x_1) \}\subset \Omega_\eps$ then, 
$$u_\eps \to u_0 \quad s-H^1(\Omega_0)$$

\end{theorem}

\begin{remark}
Functions $b$ and $G_0$, defined in (H1) and \eqref{FG0T} respectively, are related to the part of the domain $\Omega^\epsilon$ that does not oscillate as the parameter $\epsilon$ goes to zero.
Indeed, if we assume that the period, the amplitude and the profile of the domain are constant with respect to $x \in (0,1)$, 
we get the nice result stated in the corollary below.
\end{remark}

\begin{corollary}\label{periodic}
If we have $G(x,y)=G(y)$ an $L$-periodic function with $\min_{y \in \R} G(y)=0$ and $b(x)=b$ a constant function, then, the homogenized limit is given by the equation with constant coefficients:
\begin{equation} \label{homogenized-periodic}
\left\{
\begin{gathered}
 - d u''+ u=f, \quad (0,1) \\
u'(0)=u'(1)=0\end{gathered}
\right.
\end{equation}
where the diffusion coefficient is given by 
$$d=\frac{b}{b+\frac{1}{L}\int_0^LG(y) \, dy}=\frac{Lb}{Lb+\int_0^LG(y) \, dy}=\frac{\hbox{Area of the non oscillating part of the unit cell}}{
\hbox{Total area of the unit cell}}.$$
\end{corollary}

\par\bigskip\bigskip\noindent{\bf Type II.  Comb-like thin domain}. 
We consider now another interesting type of thin domain.  Let 
$R^\eps= \rm{Int} \left( \overline{R_-^\eps \cup  R^\eps_+} \right)$ 
where
$$R_-^\eps=\{ (x_1,x_2) \; | \;  0<x_1<1, -\eps b(x_1)<x_2<0\},$$ 
with $b$ given as in the previous case  (see hypothesis (H1)).  Consider also, 
$$R^\eps_+=\bigcup_{n=1}^{N_\eps}R^\eps_{n,+},$$
where $N_\eps$ is the largest integer number such that $N_\eps L\eps^\alpha\leq 1$ and 
$$R^\eps_{n,+}=\{ ((n-1)L\eps^\alpha+\eps^\alpha x_1,\eps x_2) \; | \; (x_1,x_2)\in Q_0\}$$
where $Q_0\subset (0,L)\times (0, G)$ is a fixed Lipschitz domain satisfying the following: 

\begin{itemize}
\item[(HQ)] $\partial Q_0\cap (\{0\}\times (0, G])=\emptyset$, $\partial Q_0\cap (\{L\}\times (0, G])=\emptyset$. Moreover, if $\Gamma_0 = \partial Q_0 \cap \{x_2=0\} $ and if we consider $e_1(Q_0)$ the first eigenvalue of the operator $-\Delta$ in $Q_0$ with homogeneous Dirichlet boundary condition in $\Gamma_0$ and homogeneous Neumann boundary condition in $\partial Q_0\setminus \Gamma_0$, then $e_1(Q_0)>0$. 
\end{itemize}

Observe that if $Q_0$ is connected and $\Gamma_0\ne \emptyset$ then (HQ) is satisfied. But there are cases where $Q_0$ is disconnected ant still (HQ) holds, see Figure \ref{typeII}.

As we have done in the previous case, let us define $\Omega^\eps=\{(x_1,x_2) \; | \; (x_1,\eps x_2)\in R^\eps\}$ so that,

$$\Omega^\eps= \rm{Int} \left( \overline{\Omega_- \cup  \Omega^\eps_+} \right),$$
$$\Omega_-=\{ (x_1,x_2) \; | \;  0<x_1<1, -b(x_1)<x_2<0\},$$ 
$$\Omega^\eps_+=\cup_{n=1}^{N_\eps}\Omega^\eps_{n,+},$$
where 
$$\Omega^\eps_{n,+}=\{ ((n-1)L\eps^\alpha+\eps^\alpha x_1,x_2) \; | \; (x_1,x_2)\in Q_0\}.$$
We also consider the equivalent linear elliptic problem
\begin{equation} \label{P-typeII}
\left\{
\begin{gathered}
 - \frac{\partial^2 u^\epsilon}{{\partial x_1}^2} - 
\frac{1}{\epsilon^2} \frac{\partial^2 u^\epsilon}{{\partial x_2}^2} + u^\epsilon = f^\epsilon
\quad \textrm{ in } \Omega^\epsilon \\
\frac{\partial u^\epsilon}{\partial x_1} \nu_1^\epsilon + \frac{1}{\epsilon^2} \frac{\partial u^\epsilon}{\partial x_2}\nu_2^\epsilon = 0
\quad \textrm{ on } \partial \Omega^\epsilon
\end{gathered}
\right.
\end{equation}
Under this conditions, we may get the following result.

\begin{theorem} \label{main2}
Let $u^\epsilon$ be the unique solution of (\ref{P}).
Assume that  $f^\epsilon \in L^2(\Omega^\epsilon)$ satisfies $\|f^\eps\|_{L^2(\Omega^\eps)}\leq C$
 and the function $\hat f^\eps(x)=\int_{S^\eps(x)}f^\eps(x,y) \, dy$ satisfies that $\hat f^\eps\rightharpoonup \hat f$, w-$L^2(0,1)$ where  $S^\eps(x)=\{y \; | \; (x,y)\in \Omega^\eps\}$, that is, the section of the domain $\Omega^\eps$ at the point $x\in (0,1)$. 
 
Then, if  $u_0(x_1)$ is the unique 
weak solution of the Neumann problem
$$
\int_0^1 \Big\{ b(x) \; u_x(x) \, \varphi_x(x) 
+ q(x) \, u(x) \, \varphi(x) \Big\} dx =  \int_{0}^1  \hat f(x) \, \varphi \, dx, \quad \forall \varphi\in H^1(0,1)
$$
where $q(x)$ is the function given by 
$$
q(x) = \frac{|Q_0|}{L} + b(x_1) \quad \forall x_1 \in (0,1),
$$
we have $$
\| u^\epsilon-u_0\|_{L^2(\Omega^\eps)} \eto 0.
$$ 
Moreover,
 $$u_\eps \to u_0 \quad s-H^1(\Omega_-).$$ 

\end{theorem}

\section{An important estimate} \label{basic-section}

In this section we show several basic estimates on the solutions of certain elliptic pde's posed in rectangles of the type  
$$Q_\eps=\{ (x,y) \in \R^2 \; | \;  -\eps^\alpha<x<\eps^\alpha, \, 0<y<1\}$$ 
with $\alpha>1$.  
As a matter of fact, 
for $u_0(\cdot)\in H^1(-\eps^\alpha,\eps^\alpha)$, we define the function $u^\eps(x,y)$  as the unique solution of
\begin{equation} \label{P-basic}
\left\{
\begin{gathered}
 - \frac{\partial^2 u^\epsilon}{{\partial x}^2} - 
\frac{1}{\epsilon^2} \frac{\partial^2 u^\epsilon}{{\partial y}^2}= 0
\quad \textrm{ in } Q_\epsilon, \\
\qquad u(x,0)=u_0(x),\quad  \textrm{ on } \Gamma_\eps,\\
\frac{\partial u}{\partial \nu}=0,\quad  \textrm{ on } \partial Q_\eps\setminus \Gamma_\eps
\end{gathered}
\right.
\end{equation}
where $\nu$ is the outward unit normal to $\partial Q_\eps$ and 
$$
\Gamma_\eps = \{ (x,0) \in \R^2 \, | \, -\eps^\alpha<x<\eps^\alpha \}.
$$

We have the following,

\begin{lemma}\label{basic-lemma}
With the notations above, if we denote by $\bar u_0$ the average of $u_0$ in $\Gamma_\eps$, that is 
\begin{equation}\label{u-bar}
\bar u_0=\frac{1}{2\eps^\alpha}\int_{-\eps^\alpha}^{\eps^\alpha}u_0(x) \, dx,
\end{equation}
then, there exists a constant $C$,  independent of $\eps$ and $u_0$, such that 
\begin{equation}\label{estimate-L2-in-x}
\int_{-\eps^\alpha}^{\eps^\alpha}|u^\eps(x,y)-\bar u_0|^2 \, dx \leq C \exp\left\{-\frac{2y\pi}{\eps^{\alpha-1}}\right\} \|u_0\|_{L^2(-\eps^\alpha,\eps^\alpha)}^2
\end{equation}
\begin{equation}\label{estimate-L2}
\int_0^1\int_{-\eps^\alpha}^{\eps^\alpha}|u(x,y)-\bar u_0|^2 \, dxdy \leq C\eps^{\alpha-1} \|u_0\|_{L^2(-\eps^\alpha,\eps^\alpha)}^2
\end{equation}
\end{lemma}
and 
\begin{equation}\label{basic-estimate}
\left\|\frac{\partial u}{\partial x}\right\|_{L^2(Q_\eps)}^2+\frac{1}{\eps^2}\left\|\frac{\partial u}{\partial y}\right\|^2_{L^2(Q_\eps)}\leq C \eps^{\alpha -1}
\left\|\frac{\partial u_0}{\partial x}\right\|_{L^2(-\eps^\alpha,\eps^\alpha)}^2.
\end{equation}

\begin{proof}  The proof of this result is  based in the known fact that the solution of the problem above can be found explicitly and admits a Fourier decomposition of the form 
\begin{equation}\label{solution}
u^\eps(x,y)=\frac{1}{2\eps^\alpha}\int_{-\eps^\alpha}^{\eps^\alpha} u_0(x)dx+ \sum_{  k=1  }^\infty (u_0,\varphi_n^\eps)\varphi_n^\eps(x) \frac{\cosh(\frac{n\pi(1-y)}{\eps^{\alpha-1}})}{\cosh(\frac{n\pi}{\eps^{\alpha-1}})}
\end{equation}
where $\varphi_n^\eps(x)=\eps^{-\alpha/2}\cos(\frac{n\pi x}{\eps^\alpha}),$ $n=1,2, \ldots,$ and  $(u_0,\varphi_n^\eps)=(u_0,\varphi_n^\eps)_{L^2(-\eps^\alpha, \eps^\alpha)}$. \end{proof}

\begin{remark} Observe that in particular, estimate \eqref{basic-estimate} implies  that 
$$\min_{u \in V}
\left\{ \left\|\frac{\partial u}{\partial x}\right\|_{L^2(Q_\eps)}^2+\frac{1}{\eps^2}\left\|\frac{\partial u}{\partial y}\right\|^2_{L^2(Q_\eps)}\right\} \leq C \eps^{\alpha -1}
\left\|\frac{\partial u_0}{\partial x}\right\|_{L^2(-\eps^\alpha,\eps^\alpha)}^2$$
where
$
V = \{ u \in H^1(Q_\eps) \, | \, u = u_0 \textrm{ in } \Gamma_\eps \}.
$
\end{remark}

\section{Thin domains as a region between graphs}
\label{thin-domain-graphs}

In this section we consider Type I thin domains and provide a proof of Theorem  \ref{main}.

We will start analyzing in detail the structure of the domain $\Omega^\eps$ as a preparation for the proof of our result. 

\subsection{The one parameter family $G^\eps$}
In this subsection we obtain some properties and a convenient approximation to the parameter family $G_\epsilon$ that we will use
in the proof of the main result Theorem \ref{main}.

From $(H2)$ we have that there exists a positive constant $G_1$ such that 
\begin{equation} \label{HG}
\begin{gathered}
0 \le G_\epsilon(x) \le G_1, \quad  \forall x\in (0,1), \quad \forall \eps\in (0,\eps_0).
\end{gathered}
\end{equation}
Moreover, for each $x\in [0,1]$, we consider the function  
\begin{equation} \label{FG0}
G_{0}(x) = \min_{y \in \R} G(x,y)\geq 0.
\end{equation}
We show that $G_0(\cdot)$ is  a continuous function in $[0,1]$. 
Indeed, we will prove that
\begin{equation} \label{CONTG0}
| G_0(x) - G_0(x') | \leq \sup_{y \in \R} | G(x,y) - G(x',y) | \quad \forall x, x' \in [0,1].
\end{equation}
Consequently, the continuity of $G_0$ follows from the uniform continuity of $G$ in $y$ and inequality (\ref{CONTG0}).

Thus, let us prove (\ref{CONTG0}). Given $x$ and $x' \in [0,1]$, there exist $y(x)$ and $y(x') \in \R$ such that 
$G_0(x) = G(x, y(x))$ and $G_0(x') = G(x', y(x'))$. We have
\begin{equation} \label{eqG01}
G_0(x) - G_0(x') = G_0(x) - G(x, y(x')) + G(x, y(x')) - G(x', y(x')) \leq G(x, y(x')) - G(x', y(x')).
\end{equation}
In a completely symmetric way we get 
\begin{equation} \label{eqG02}
G_0(x') - G_0(x) \leq G(x', y(x)) - G(x, y(x)).
\end{equation}
Consequently, we obtain (\ref{CONTG0}) from (\ref{eqG01}) and (\ref{eqG02}).

Recalling that we denote by $N_\eps$ the largest integer number such that $N_\eps L\eps^{\alpha}\leq 1$, where $L$ is given in hypothesis (H2), we define 
\begin{equation} \label{eqG00}
G_{n,\eps}=\min_{x\in [(n-1)L \eps^{\alpha}, nL \eps^{\alpha}]} G\left(x,\frac{x}{\eps^\alpha}\right), \quad n=1,2\ldots, N_\eps
\end{equation}
and $\gamma_{n,\eps}\in [(n-1)L \eps^{\alpha}, nL \eps^{\alpha}]$ a point where the minimum (\ref{eqG00}) is attained, 
that is, $G(\gamma_{n,\eps},\frac{ \gamma_{n,\eps}}{\eps^\alpha})=G_{n,\eps}$. Observe that $\gamma_{n,\eps}$ does not need to be uniquely defined.
We also denote by $\gamma_{0,\eps}=0$ and $\gamma_{ N_\eps+1,\eps}=1$.

Note that the set
\begin{equation} \label{PIE}
\{ \gamma_{0,\eps}, \gamma_{1,\eps}, ..., \gamma_{N_\eps + 1,\eps} \}
\end{equation}
defines a partition for the unit interval $[0,1]$.
Also, we have by definition that the segments 
$$
\{ (\gamma_{n,\eps}, x_2) \; | \; G_{n,\eps}<x_2<G_1\}\cap \Omega^\eps=\emptyset,
$$ 
for all $n=1,2,\ldots, N_\eps$. 
 
Consider also the step function 
\begin{equation} \label{PFG}
\tilde G_{0}^\eps(x)=
\left\{
\begin{array} {ll}
 G_{1,\eps}, &x\in  [0, \gamma_{1,\eps}] \\ 
  \max\{G_{n,\eps},G_{n+1,\eps}\}, & x\in  [\gamma_{n,\eps}, \gamma_{n+1,\eps}],n=1,2\ldots, N_\eps - 1   \\
G_{N_\eps,\eps} , & x\in  [\gamma_{N_\eps,\eps}, 1]
\end{array}
\right. .
\end{equation}

\begin{lemma} \label{ESTINF}
We have 
$$
\| G_0- \tilde G^\eps_0\|_{L^\infty(0,1)} \to 0 \quad \textrm{ as } \eps \to 0.
$$
\end{lemma}
\begin{proof}
It follows from $(H2)$ and (\ref{CONTG0}) that, for each $\eta > 0$, there exists $\eps_0 > 0$ such that 
\begin{equation} \label{eqGU}
|G(x,y) - G(x',y)|<\eta, |G_0(x) - G_0(x')|<\eta, \quad \forall | x - x'|< 2 \, \eps_0^\alpha \, L,\, \forall y \in \R
\end{equation}
Now, for all $x \in [\gamma_{n,\eps}, \gamma_{n+1,\eps}]$ we have
$$
\tilde G^\eps_0(x) - G_0(x) = \max \left\{ G_{n,\eps}, G_{n+1,\eps} \right\} - G_0(x).
$$
Without loss of generality, we may assume $\tilde G^\eps_0(x) = G_{n,\eps}$, that is, $G_{n,\eps} \geq G_{n+1,\eps}$. 
Thus,
\begin{eqnarray} \label{eqG0}
\tilde G^\eps_0(x) - G_0(x) & = & G_{n,\eps} - G_0(x) \nonumber \\
& = & G(\gamma_{n,\eps},\gamma_{n,\eps}/\eps^\alpha) - G_0(x) \nonumber \\
& = & G(\gamma_{n,\eps},\gamma_{n,\eps}/\eps^\alpha) - G_0(\gamma_{n,\eps}) + G_0(\gamma_{n,\eps}) - G_0(x).
\end{eqnarray}
It follows from definition of $G_0$ in \eqref{FG0}, that
$$
G(\gamma_{n,\eps},\gamma_{n,\eps}/\eps^\alpha) - G_0(\gamma_{n,\eps}) \geq 0.
$$
Also, since $G(x,\cdot)$ is $l(x)$-periodic with $|l(x)| \leq L$, we have that 
there exist $y(\gamma_{n,\eps}) \in [0,l(\gamma_{n,\eps})]$ and $k(\gamma_{n,\eps}) \in \N$ with  
$y(\gamma_{n,\eps}) + k(\gamma_{n,\eps}) \, l(\gamma_{n,\eps}) \in [(n-1) \, L \, \eps^\alpha, n \, L \, \eps^\alpha]$,
such that
\begin{equation} \label{eqG1}
G_0(\gamma_{n,\eps}) = G(\gamma_{n,\eps}, y(\gamma_{n,\eps}) )
= G(\gamma_{n,\eps}, y(\gamma_{n,\eps}) + k(\gamma_{n,\eps}) \, l(\gamma_{n,\eps}) ).
\end{equation}
Consequently, we get from (\ref{eqG00}) and (\ref{eqG1}) that
\begin{eqnarray} \label{eqG2} 
G(\gamma_{n,\eps},\gamma_{n,\eps}/\eps^\alpha) - G_0(\gamma_{n,\eps}) & = & 
G(\gamma_{n,\eps},\gamma_{n,\eps}/\eps^\alpha) - G( (y + k \, l) \, \eps^\alpha , (y + k \, l) \, \eps^\alpha / \eps^\alpha ) \nonumber \\
& & + G( (y + k \, l) \, \eps^\alpha , (y + k \, l) \, \eps^\alpha / \eps^\alpha ) - G(\gamma_{n,\eps}, (y + k \, l) \, \eps^\alpha/\eps^\alpha ) 
\nonumber \\
& \leq & G( (y + k \, l) \, \eps^\alpha , (y + k \, l) \, \eps^\alpha / \eps^\alpha ) 
- G(\gamma_{n,\eps}, (y + k \, l) \, \eps^\alpha/\eps^\alpha )
\end{eqnarray}
since 
$$
G(\gamma_{n,\eps},\gamma_{n,\eps}/\eps^\alpha) - G( (y + k \, l) \, \eps^\alpha , (y + k \, l) \, \eps^\alpha / \eps^\alpha )
\leq 0.
$$
Therefore, due to (\ref{eqG0}), (\ref{eqG2}) and (\ref{eqGU}), we obtain
\begin{eqnarray*}
| \tilde G^\eps_0(x) - G_0(x) | 
& \leq & |G( (y + k \, l) \, \eps^\alpha , (y + k \, l) \, \eps^\alpha / \eps^\alpha ) 
- G(\gamma_{n,\eps}, (y + k \, l) \, \eps^\alpha/\eps^\alpha )| + |G_0(\gamma_{n,\eps}) - G_0(x)| 
 <  2 \eta 
\end{eqnarray*}
whenever $x \in [\gamma_{n,\eps},\gamma_{n+1,\eps}]$. 

Then, since $x \in [0,1]$ is arbitrary and $\cup_{n=1}^{N_\eps} [\gamma_{n,\eps},\gamma_{n+1,\eps}] = [0,1]$, 
we conclude the proof.
\end{proof}

The following result will also be needed.  
\begin{lemma}
We have the following
\begin{equation} \label{LGE}
G_\eps(\cdot)  \etow  \frac{1}{l(\cdot)}\int_0^{l(\cdot)}G(\cdot,s)ds \qquad  w^*-L^\infty(0,1).
\end{equation}
\end{lemma}
\begin{proof}
We have to prove
\begin{equation} \label{eqGE0}
\int_0^1 \left\{
G_\eps(x) - \frac{1}{l(x)} \int_0^{l(x)} G(x,s) \, ds
\right\} \varphi(x) \, dx \to 0 \quad \textrm{ as } \eps \to 0
\end{equation}
for all $\varphi \in L^1(0,1)$.  With standard density arguments, it is enough to show  (\ref{eqGE0}) when $\varphi$ is a  characteristic function.
Then, for $0 \leq e < f \leq 1$ we consider the following characteristic function
$$
\varphi(x) = 
\left\{
\begin{array}{ll}
1 & x \in (e,f) \\
0 & x \notin (e,f)
\end{array}
\right. .
$$
So, we have to estimate the integral
$$
I = \int_e^f \left\{
G_\eps(x) - \frac{1}{l(x)} \int_0^{l(x)} G(x,s) \, ds
\right\} dx
$$
as $\eps > 0$ goes to zero.

For this, let $\eta>0$ be a small number and let $\{ e=x_0, x_1, ..., x_n=f \}$ be a partition for the interval $(e,f)$, and 
$\hat x_i$ be a fixed point in the interval $J_i=[x_{i-1},x_i]$, $i=1, ..., n$, such that 
$$
\sup_i \sup_{x \in J_i, \, y \in \R}
| G(x,y) - G(\hat x_i,y) | < \eta.
$$
Observe that we can write 
$$
I = \sum_{i=1}^5 I^i
$$
where
$$
\begin{gathered}
I^1 =   
\sum_{i=1}^n \int_{J_i} \left\{ G(x,x/\eps^\alpha) - G(\hat x_i,x/\eps^\alpha) \right\} dx \\
I^2 = \sum_{i=1}^n \int_{J_i} \left\{ G(\hat x_i,x/\eps^\alpha) - 
\frac{1}{l(\hat x_i)} \int_0^{l(\hat x_i)} G(\hat x_i,s) \, ds \right\} dx \\
I^3 = \sum_{i=1}^n \int_{J_i} \left\{ \frac{1}{l(\hat x_i)} \int_0^{l(\hat x_i)} G(\hat x_i,s) \, ds - 
\frac{1}{l(\hat x_i)} \int_0^{l(\hat x_i)} G(x,s) \, ds \right\} dx \\
I^4 = \sum_{i=1}^n \int_{J_i} \left\{ \frac{1}{l(\hat x_i)} \int_0^{l(\hat x_i)} G(x,s) \, ds - 
\frac{1}{l(\hat x_i)} \int_0^{l(x)} G(x,s) \, ds \right\} dx \\
I^5 = \sum_{i=1}^n \int_{J_i} \left\{ \frac{1}{l(\hat x_i)} \int_0^{l(x)} G(x,s) \, ds - 
\frac{1}{l(x)} \int_0^{l(x)} G(x,s) \, ds \right\} dx .
\end{gathered}
$$

It is easy to estimate the integrals $I^1$, $I^3$, $I^4$ and $I^5$ to obtain
\begin{equation} \label{ESTINTS}
\begin{gathered}
|I^1| \leq \eta \, (f-e) \\
|I^3| \leq \eta \, (f-e) \\
|I^4| \leq G_1 \, \| \hat l^\eta - l \|_{L^\infty(0,1)} \, (f-e) \\
|I^5| \leq G_1 \, \left( L/{L'}^2 \right) \, \| \hat l^\eta - l \|_{L^\infty(0,1)} \, (f-e)
\end{gathered}
\end{equation}
where $G_1$, $L$ and $L'$ are the positive constants given by hypothesis $(H2)$, and 
the function $\hat l^\eta$ is the step function defined for each $\eta > 0$ by 
$$
\hat l^\eta(x) = l(x_i) \quad \textrm{ as } x_i \in J_i.
$$
Since inequalities (\ref{ESTINTS}) do not depend on $\eps > 0$, and
$\| \hat l^\eta - l \|_{L^\infty(0,1)} \to 0$ as $\eta \to 0$ uniformly in $\eps$,
we have that $I^1$, $I^3$, $I^4$ and $I^5$ go to zero as $\eta \to 0$
uniformly in $\eps > 0$.

Hence, to conclude the proof, we just evaluate the integral $I^2$.
But this is a straightforward application of the Average Theorem 
since $\hat x_i$ is a fixed point in $J_i$, and $G(\hat x_i,\cdot)$ is a 
$l(\hat x_i)$-periodic function. Indeed,
$$
I^2 = \sum_{i=1}^n \int_{J_i} \left\{ G(\hat x_i,x/\eps^\alpha) - 
\frac{1}{l(\hat x_i)} \int_0^{l(\hat x_i)} G(\hat x_i,s) \, ds \right\} dx
\to 0 \textrm{ as } \eps \to 0.
$$
\end{proof}

\subsection{Proof of Theorem \ref{main}} \label{PofC}

Here, we give a proof of Theorem \ref{main}.

\begin{proof}
The variational formulation of (\ref{P}) is:  find $u^\epsilon \in H^1(\Omega^\epsilon)$ such that 
\begin{equation} \label{VFP}
\int_{\Omega^\epsilon} \Big\{ \frac{\partial u^\epsilon}{\partial x_1} \frac{\partial \varphi}{\partial x_1} 
+ \frac{1}{\epsilon^2} \frac{\partial u^\epsilon}{\partial x_2} \frac{\partial \varphi}{\partial x_2}
+ u^\epsilon \varphi \Big\} dx_1 dx_2 = \int_{\Omega^\epsilon} f^\epsilon \varphi dx_1 dx_2, 
\quad \forall \varphi \in H^1(\Omega^\epsilon).
\end{equation}
Taking $\varphi = u^\epsilon$ in  expression (\ref{VFP}) and using that $\|f^\eps\|_{L^2(\Omega^\eps)}\leq C$,  we easily obtain the a priori bounds
\begin{equation} \label{EST0}
\begin{gathered}
\| u^\epsilon \|_{L^2(\Omega^\epsilon)}, \Big\| \frac{\partial u^\epsilon}{\partial x_1} \Big\|_{L^2(\Omega^\epsilon)}
\textrm{ and } \frac{1}{\epsilon} \Big\| \frac{\partial u^\epsilon}{\partial x_2} \Big\|_{L^2(\Omega^\epsilon)} 
\le C.
\end{gathered}
\end{equation}
In particular, we have 
$$
\Big\| \frac{\partial u^\epsilon}{\partial x_2} \Big\|_{L^2(\Omega^\epsilon)} \leq \eps \, C \to 0 \textrm{ as } \eps \to 0.
$$

Let us observe that domain $\Omega^\eps$ consists of two main parts. One of them is a highly oscillating domain $\Omega^\eps_+$ and 
the other one is a non oscillating domain $\Omega^\eps_-$.
To define these domains, we use the step function $\tilde G_{0}^\eps$ defined in (\ref{PFG}). 
So, we consider the following open sets 
\begin{equation} \label{DOMAINS}
\begin{gathered}
\Omega^\eps_- = \{ (x_1, x_2) \in \R^2 \, | \, x_1 \in (0,1), \, -b(x_1) < x_2 < \tilde G_{0}^\eps(x_1) \} \\
\Omega^\eps_+ = \{ (x_1, x_2) \in \R^2 \, | \, x_1 \in (0,1), \, \tilde G_{0}^\eps(x_1) < x_2 < G_\eps(x_1) \}.
\end{gathered}
\end{equation}
Notice that
$$
\begin{gathered}
\Omega^\eps = \mbox{Int} \left( \overline{\Omega^\eps_+ \cup \Omega^\eps_-} \right).
\end{gathered}
$$

We will also need to consider the domain
$$
\Omega_0 = \{ (x_1, x_2) \in \R^2 \, | \, x_1 \in (0,1), \, -b(x_1) < x_2 < G_0(x_1) \},
$$
where $G_0$ is defined by (\ref{FG0}).  Notice that we have $\Omega_0\subset \Omega^\eps$. 

We want to pass to the limit in the variational formulation (\ref{VFP}) for certain appropriately chosen test functions. 
In order to accomplish this, we rewrite it as follows
\begin{eqnarray} \label{VFP2}
& & \int_{\Omega^\epsilon_+} \Big\{ \frac{\partial u^\epsilon}{\partial x_1} \frac{\partial \varphi}{\partial x_1} 
+ \frac{1}{\epsilon^2} \frac{\partial u^\epsilon}{\partial x_2} \frac{\partial \varphi}{\partial x_2} \Big\} dx_1 dx_2
+ \int_{\Omega^\eps_-} \Big\{ \frac{\partial u^\epsilon}{\partial x_1} \frac{\partial \varphi}{\partial x_1} 
+ \frac{1}{\epsilon^2} \frac{\partial u^\epsilon}{\partial x_2} \frac{\partial \varphi}{\partial x_2} \Big\} dx_1 dx_2
\nonumber \\
& & \quad \quad + \int_{\Omega^\epsilon} u^\epsilon \varphi \, dx_1 dx_2 = \int_{\Omega^\epsilon} f^\epsilon \varphi \, dx_1 dx_2, 
\, \forall \varphi \in H^1(\Omega^\epsilon).
\end{eqnarray}
Now, we pass to the limit in the different functions that form the integrands of \eqref{VFP2}.

\par\bigskip\noindent {\bf  (a). Limit of $u^\eps$ in $L^2(\Omega^\epsilon)$}. 


It follows from (\ref{EST0}) that $u^\eps|_{\Omega_0} \in H^1(\Omega_0)$ and satisfies for all $\epsilon > 0$
$$
\| u^\epsilon \|_{L^2(\Omega_0)}, \Big\| \frac{\partial u^\epsilon}{\partial x_1} \Big\|_{L^2(\Omega_0)}
\textrm{ and } \frac{1}{\epsilon} \Big\| \frac{\partial u^\epsilon}{\partial x_2} \Big\|_{L^2(\Omega_0)} 
\le C.
$$
Then, we can extract a subsequence of $\{ u^\eps|_{\Omega_0} \} \subset H^1(\Omega_0)$, denoted again by $u^\epsilon$, such that
\begin{equation} \label{WC0}
\begin{gathered}
u^\epsilon \rightharpoonup u_0 \quad w-H^1(\Omega_0) \\
u^\epsilon \rightarrow u_0 \quad s-H^s(\Omega_0) \textrm{ for all } s \in [0, 1) \textrm{ and }\\
\frac{\partial u^\epsilon}{\partial x_2} \rightarrow 0 \quad s-L^2(\Omega_0) \\
\end{gathered}
\end{equation}
as $\epsilon \to 0$ for some  $u_0 \in H^1(\Omega_0)$.

A consequence of the limits (\ref{WC0}) is that $ u_0(x_1,x_2)$ does not depend on the variable $x_2$.  
More precisely,
\begin{equation} \label{U0}
\frac{\partial  u_0}{\partial x_2}(x_1,x_2) = 0 \textrm{ a.e. } \Omega_0.
\end{equation}
Also, due to (\ref{WC0}), we have that the restriction of $u^\eps$ to the coordinate axis $x_1$ converges to $u_0$.
That is, if $\Gamma = \{ (x_1,0) \in \R^2 \, | \, x_1 \in (0,1) \}$, then 
\begin{equation} \label{TRACE}
u^\eps|_{\Gamma} \rightarrow u_0 \quad s-H^s(\Gamma), \quad \forall s \in [0,1/2)
\end{equation}

Now, we can see that (\ref{TRACE}) with $s=0$,  implies the $L^2$-convergence of $u^\eps$ to $u_0$, that is  
\begin{equation} \label{L2CONV}
\| u^\eps - u_0 \|_{L^2(\Omega^\eps)} \to 0 \textrm{ as } \eps \to 0.
\end{equation}
In fact, it follows from (\ref{TRACE}) that
\begin{eqnarray*}
& &\|u^\eps(\cdot,0)-u_0\|^2_{L^2(\Omega_\eps)}  = 
\int_0^1 \int_{-b(x_1)}^{G_\eps(x_1)} | u^\eps(x_1,0) - u_0(x_1) |^2 \, dx_2 dx_1 \\
& &\qquad\leq  C(G,b)\, \| u^\eps - u_0 \|_{L^2(\Gamma)} \to  0 \textrm{ as } \eps \to 0
\end{eqnarray*}
where $C(G,b)=\|b\|_{L^\infty}+G_1$.
Also,
$$
u^\eps(x_1,x_2) - u^\eps(x_1,0) = \int_0^{x_2} \frac{\partial u^\eps}{\partial x_2}(x_1,s) \, ds
$$
and with H\"older inequality, 
$$
| u^\eps(x_1,x_2) - u^\eps(x_1,0) |^2 \leq \left( \int_0^{x_2} \left| \frac{\partial u^\eps}{\partial x_2}(x_1,s) \right|^2 ds \right) \, |x_2|.
$$
Hence, integrating in $\Omega^\eps$ and using (\ref{EST0}) to get
\begin{eqnarray*}
& & \|u^\eps - u^\eps(\cdot,0)\|_{L^2(\Omega_\eps)}^2= 
\int_0^1 \int_{-b(x_1)}^{G_\eps(x_1)} | u^\eps(x_1,x_2) - u^\eps(x_1,0) |^2 \, dx_1 dx_2 \\
& &\quad \leq  \int_0^1 \int_{-b(x_1)}^{G_\eps(x_1)} \left(  \int_0^{x_2} \left| \frac{\partial u^\eps}{\partial x_2}(x_1,s) \right|^2 ds \right) 
\, |x_2| \, dx_2 dx_1 \\ & & \qquad \leq  C(G,b) \, \left\| \frac{\partial u^\eps}{\partial x_2} \right\|_{L^2(\Omega^\eps)}^2 \leq  C(G,b)\cdot C\eps^2  \to 0 \textrm{ as } \eps \to 0.
\end{eqnarray*}

Therefore,
\begin{eqnarray*}
\| u^\eps - u_0 \|_{L^2(\Omega_\eps)}  & \leq & \|u^\eps - u^\eps(\cdot,0)\|_{L^2(\Omega_\eps)}+ 
 \|u^\eps(\cdot,0) - u_0\|_{L^2(\Omega_\eps)} \to 0 \textrm{ as } \eps \to 0.
\end{eqnarray*}

\par\bigskip\noindent {\bf  (b). Limit of $f^\eps$}. 

Since $\|f^{ \eps}\|_{L^2(\Omega^{ \eps})}\leq C$, with $C$ independent of $\eps$, we have that the function $\hat f^\eps$ defined by 
\begin{equation} \label{FHF}
\hat f^\eps (x_1) \equiv \int_{-b(x_1)}^{G_\eps(x_1)} f^\eps(x_1,x_2)dx_2
\end{equation}
belongs to $L^2(0,1)$ and satisfies 
$
\| \hat f^\eps \|_{L^2(0,1)} \leq C 
$ 
for some constant $C$ independent of $\eps$ also.
Hence, via subsequences, we have the existence of a function 
$\hat f=\hat f(x_1)\in L^2(0,1)$ such that 
\begin{equation}\label{LIMITF}
\hat f^\eps \rightharpoonup \hat f \qquad w-L^2(0,1).
\end{equation}

\begin{remark}\label{remark-1}
Observe that in the case where $f^\eps(x_1,x_2)=f(x_1)$ then  
\begin{eqnarray*}
\hat f^\eps(x_1) = \left( G(x_1,x_1/\eps^\alpha) + b(x_1)  \right) \, f(x_1) \rightharpoonup  p(x_1) \, f(x_1) \qquad w^*-L^\infty(0,1)
\end{eqnarray*}
where the function $p$ is given by
\begin{equation} \label{PF}
p(x) = \frac{1}{l(x)} \int_0^{l(x)} G(x,s) \, ds + b(x),
\end{equation}
and observe that $ \frac{1}{l(x)} \int_0^{l(x)} G(x,s)ds$ is the weak $*$-$L^\infty(0,1)$ limit of $G_\eps(x)$ obtained in (\ref{LGE}). 
Consequently, we have that 
$$
\hat f(x) = p(x) \, f(x) \quad x \in (0,1).
$$
\end{remark}

\par\bigskip\noindent {\bf  (c). Test functions}.

Here, we define suitable test functions that will allow us to pass to the limit in the variational formulation (\ref{VFP2}). 
For this, we use the definition of the open sets $\Omega^\eps_-$ and $\Omega^\eps_+$ given in (\ref{DOMAINS}).

For each $\phi \in H^1(0,1)$ and $\eps > 0$, 
we define the following test functions in $H^1(\Omega^\eps)$
\begin{equation} \label{TESTF}
\begin{gathered}
\varphi^\eps(x_1,x_2) = \left\{
\begin{array}{ll}
X^\eps_n(x_1,x_2), & (x_1,x_2) \in \Omega^\eps_+\cap Q^\eps_n,\quad n=1,2,\ldots \\
\phi(x_1), & (x_1,x_2) \in \Omega_-^\eps 
\end{array}
\right.
\end{gathered}
\end{equation}
where  $Q^\eps_n$ is the rectangle (see Figure 4)
$$
Q^\eps_n=\{ (x_1,x_2) \, | \, \gamma_{n,\eps}<x_1<\gamma_{n+1,\eps}, \, \tilde G_{0}^\eps(x_1)<x_2<G_1 \}
$$
and the function $X^\eps_n$ is the solution of the problem
\begin{equation} \label{AUXSOL}
\left\{
\begin{gathered}
- \frac{\partial^2 X^\eps}{\partial x_1^2} - \frac{1}{\eps^2} \frac{\partial^2 X^\eps}{\partial x_2^2} 
= 0, \quad \textrm{ in } Q^\eps_n \\
\frac{\partial X^\eps}{\partial N^\eps}=0, \quad \textrm{ on }  \partial Q^\eps_n \backslash \Gamma_n^\eps  \\
X^\eps(x_1,x_2) = \phi(x_1),  \quad \textrm{ on } \Gamma_n^\eps
\end{gathered}
\right. 
\end{equation}
where $\Gamma_n^\eps$ is the base of the rectangle, that is, 
$$
\Gamma_n^\eps = \{ (x_1, \tilde G_0^\eps(x_1)) \, | \, \gamma_{n,\eps}\leq x_1\leq \gamma_{n+1,\eps}\}.
$$

\begin{figure}[htp]\label{Figure4}
\centering \scalebox{0.5}{\includegraphics{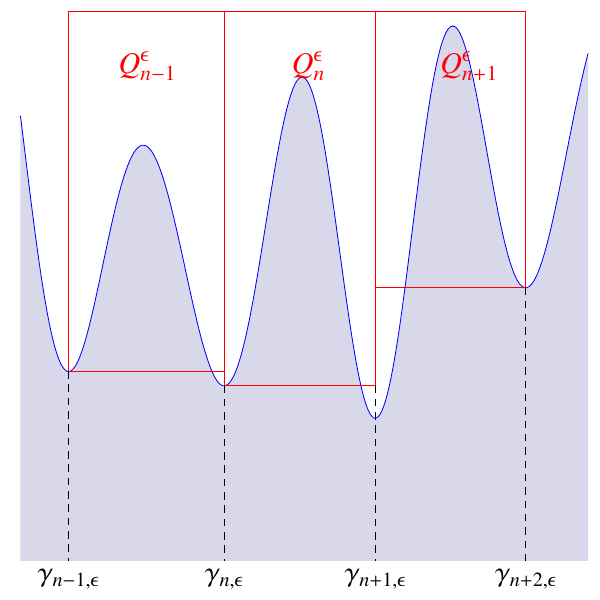}}
\caption{Rectangle $Q^\eps_n$.}
\end{figure}

It follows from estimate (\ref{basic-estimate}) that 
\begin{equation} \label{ESTX}
\left\| \frac{\partial X^\eps}{\partial x_1} \right\|^2_{L^2(Q^\eps_n)}
+ \frac{1}{\eps^2} \left\| \frac{\partial X^\eps}{\partial x_2} \right\|^2_{L^2(Q^\eps_n)}
\leq C \eps^{\alpha-1} \| \phi' \|^2_{L^2(\gamma_{n,\eps},\gamma_{n+1,\eps})}.
\end{equation}
If we denote by  $Q^\eps = \cup_{i=1}^{N_\eps} Q^\eps_n$, we have $\Omega^\eps_+= Q^\eps \cap \Omega^\eps$. Hence, 
\begin{equation} \label{ESTX2}
\begin{array}{l}
\displaystyle \left\| \frac{\partial \varphi^\eps}{\partial x_1} \right\|^2_{L^2(\Omega^\eps_+)}
+ \frac{1}{\eps^2} \left\| \frac{\partial \varphi^\eps}{\partial x_2} \right\|^2_{L^2(\Omega^\eps_+)}
= \sum^{N_\eps}_{i=1} \left( 
\left\| \frac{\partial \varphi^\eps}{\partial x_1} \right\|^2_{L^2(Q^\eps_n)}
+ \frac{1}{\eps^2} \left\| \frac{\partial \varphi^\eps}{\partial x_2} \right\|^2_{L^2(Q^\eps_n)}
\right) \\
\displaystyle \qquad \quad\leq \sum^{N_\eps}_{i=1} C \, \eps^{\alpha -1} \, \left\| \phi' \right\|^2_{L^2(\gamma_{\eps,n},\gamma_{n+1,\eps})}  \leq  C \, \eps^{\alpha -1} \left\| \phi' \right\|^2_{L^2(0,1)}.
\end{array}
\end{equation}
Furthermore, we can show that
\begin{equation} \label{TFCONV}
\| \varphi^\eps - \phi \|_{L^2(\Omega^\eps)} \to 0 \textrm{ as } \eps \to 0.
\end{equation}
We can argue as in (\ref{L2CONV}). Indeed, since
$$
\varphi^\eps(x_1,x_2) - \phi(x_1) = \varphi^\eps(x_1,x_2) - \varphi^\eps(x_1,0) 
= \int_0^{x_2} \frac{\partial \varphi^\eps}{\partial x_2}(x_1,s) \, ds,
$$
we have by (\ref{TESTF}) and (\ref{ESTX2}) that
\begin{eqnarray*}
\| \varphi^\eps - \phi \|_{L^2(\Omega^\eps)}^2  \leq  C(G,b) \, \left\| \frac{\partial \varphi^\eps}{\partial x_2} \right\|_{L^2(\Omega^\eps)}^2
 \le  C\cdot  C(G,b) \, \eps^{1+\alpha} \, \left\| \phi' \right\|^2_{L^2(0,1)} 
 \to  0 \textrm{ as } \eps \to 0.
\end{eqnarray*}

\par\bigskip\noindent {\bf  (d). Passing to the limit in the weak formulation}.

Now, we pass to the limit in the variational formulation of the problem using the test functions $\varphi^\eps$ defined above.
For this, we analyze the different functions that form the integrands in \eqref{VFP2}.

\begin{itemize}
\item First integrand: we claim that
\begin{equation} \label{INT1}
\int_{\Omega^\epsilon_+} \Big\{ \frac{\partial u^\epsilon}{\partial x_1} \frac{\partial \varphi^\eps}{\partial x_1} 
+ \frac{1}{\epsilon^2} \frac{\partial u^\epsilon}{\partial x_2} \frac{\partial \varphi^\eps}{\partial x_2} \Big\} dx_1 dx_2
\to 0 \textrm{ as } \eps \to 0.
\end{equation}

Indeed, it follows from (\ref{ESTX2}) and $\alpha > 0$ that
\begin{eqnarray} \label{eqd1}
& & \qquad \qquad\qquad \big|\int_{\Omega^\epsilon_+} \Big\{ \frac{\partial u^\epsilon}{\partial x_1} \frac{\partial \varphi^\eps}{\partial x_1} 
+ \frac{1}{\epsilon^2} \frac{\partial u^\epsilon}{\partial x_2} \frac{\partial \varphi^\eps}{\partial x_2} \Big\} dx_1 dx_2\big| \nonumber  
\\
& & \quad \leq  \left( \int_{\Omega^\eps_u} \Big\{ \left( \frac{\partial u^\epsilon}{\partial x_1} \right)^2
+ \frac{1}{\epsilon^2} \left( \frac{\partial u^\epsilon}{\partial x_2} \right)^2 \Big\} dx_1 dx_2 \right)^{1/2}
\left( \int_{\Omega^\eps_+} \Big\{ \left( \frac{\partial X^\epsilon}{\partial x_1} \right)^2
+ \frac{1}{\epsilon^2} \left( \frac{\partial X^\epsilon}{\partial x_2} \right)^2 \Big\} dx_1 dx_2 \right)^{1/2} \nonumber \\
& & \quad \leq  C \, \eps^{(\alpha - 1)/2} \,  \| u^\eps \|_{H^1(\Omega^\eps)} \, \| \phi' \|_{L^2(0,1)} \to 0 \textrm{ as } \eps \to 0. 
\end{eqnarray}

\par\medskip 
\item Second integrand: we aim
\begin{equation} \label{INT2}
\int_{\Omega^\eps_-} \Big\{ \frac{\partial u^\epsilon}{\partial x_1} \frac{\partial \varphi^\eps}{\partial x_1} 
+ \frac{1}{\epsilon^2} \frac{\partial u^\epsilon}{\partial x_2} \frac{\partial \varphi^\eps}{\partial x_2} \Big\} dx_1 dx_2
\to \int_0^1 \left( G_0(x_1) + b(x_1) \right) \, u'_0(x_1) \,  \phi'(x_1) \, dx_1 \textrm{ as } \eps \to 0.
\end{equation}
To prove this, observe that using (\ref{TESTF}), we obtain
$$
\frac{\partial \varphi^\eps}{\partial x_1} \Big|_{\Omega^\eps_-} = \frac{\partial \phi}{\partial x_1} = \phi' \quad
\textrm{ and } \quad 
\frac{\partial \varphi^\eps}{\partial x_2} \Big|_{\Omega^\eps_-} = \frac{\partial \phi}{\partial x_2} = 0
$$ 
for all $\eps > 0$. Hence, we have that
\begin{eqnarray} \label{eqSI}
& & \int_{\Omega^\eps_-} \Big\{ \frac{\partial u^\epsilon}{\partial x_1} \frac{\partial \varphi^\eps}{\partial x_1} 
+ \frac{1}{\epsilon^2} \frac{\partial u^\epsilon}{\partial x_2} \frac{\partial \varphi^\eps}{\partial x_2} \Big\} dx_1 dx_2
 =  \int_{\Omega^\eps_-} \frac{\partial u^\epsilon}{\partial x_1}(x_1,x_2) \,  \phi'(x_1) \, dx_1 dx_2  \nonumber \\
& & \qquad \quad \quad = \int_{\Omega_0} \frac{\partial u^\epsilon}{\partial x_1}(x_1,x_2) \,  \phi'(x_1) \, dx_1 dx_2 
- \int_{\Omega_0 \backslash \Omega^\eps_-} \frac{\partial u^\epsilon}{\partial x_1}(x_1,x_2) \,  \phi'(x_1) \, dx_1 dx_2 \nonumber \\
& & \qquad \quad \quad + \int_{\Omega^\eps_- \backslash \Omega_0} \frac{\partial u^\epsilon}{\partial x_1}(x_1,x_2) \,  \phi'(x_1) \, dx_1 dx_2.
\end{eqnarray}
Due to (\ref{WC0}), we can pass to the limit as $\eps \to 0$ in the first integral on the right hand side of (\ref{eqSI}) to obtain
\begin{eqnarray*} 
\int_{\Omega_0} \frac{\partial u^\epsilon}{\partial x_1}(x_1,x_2) \,  \phi'(x_1) \, dx_1 dx_2
& \to & \int_{\Omega_0} u'_0(x_1) \,  \phi'(x_1) \, dx_1 dx_2.   
\end{eqnarray*}
Also, we have that
\begin{eqnarray} \label{eqINT21}
\int_{\Omega_0} u'_0(x_1) \,  \phi'(x_1) \, dx_1 dx_2
& = & \int_0^1 u'_0(x_1) \,  \phi'(x_1) \, \left( \int_{-b(x_1)}^{G_0(x_1)} dx_2 \right)dx_1 \nonumber \\
& = & \int_0^1 \left( G_0(x_1) + b(x_1) \right) \, u'_0(x_1) \,  \phi'(x_1) \, dx_1.
\end{eqnarray}

Now, we will get (\ref{INT2}) if we prove that the remaining integrals of (\ref{eqSI}) tend to zero as $\eps \to 0$.
We evaluate one of them. The computations for the other are similar. 

From (\ref{EST0}), (\ref{DOMAINS}) and Remark \ref{ESTINF}, we have that
\begin{eqnarray} \label{eqINT22}
\int_{\Omega^\eps_- \backslash \Omega_0} \frac{\partial u^\epsilon}{\partial x_1}(x_1,x_2) \,  \phi'(x_1) \, dx_1 dx_2
& \leq & \left\| \frac{\partial u^\epsilon}{\partial x_1} \right\|_{L^2(\Omega^\eps)}
\, \| \phi' \|_{L^2(\Omega^\eps_- \backslash \Omega_0)} \nonumber \\
& \leq & C \, \left\{ \int_0^1  {\phi'(x_1)}^2 \, \left| G_0(x_1) - \tilde G^\eps_0(x_1) \right| \, dx_1 \right\}^{1/2} \nonumber \\
& \leq & C \, \| \phi' \|_{L^2(0,1)} \, \| G_0 -  \tilde G^\eps_0 \|_{L^\infty(0,1)}^{1/2} \nonumber \\
& \to & 0 \textrm{ as } \eps \to 0.
\end{eqnarray}

Therefore, we obtain (\ref{INT2}) from (\ref{eqINT21}) and (\ref{eqINT22}).
\par\medskip

\item Third integrand: if $p(x)$ is defined in (\ref{PF}) then,
\begin{equation} \label{INT3}
\int_{\Omega^\eps} u^\eps \, \varphi^\eps \, dx_1 dx_2 \to \int_0^1 p(x_1) \, u_0(x_1) \, \phi(x_1) \, dx_1 \textrm{ as } \eps \to 0.
\end{equation}

To prove \eqref{INT3}, observe that 
\begin{eqnarray*}
\int_{\Omega^\eps} u^\eps \, \varphi^\eps \, dx_1 dx_2 & = & 
\int_{\Omega^\eps} \left( u^\eps - u_0 \right) \, \varphi^\eps \, dx_1 dx_2
+ \int_{\Omega^\eps} u_0 \, \left( \varphi^\eps - \phi \right) \, dx_1 dx_2
+ \int_{\Omega^\eps} u_0 \, \phi \, dx_1 dx_2.
\end{eqnarray*}

From (\ref{L2CONV}) and (\ref{TFCONV}), we have 
$$
\begin{gathered}
\int_{\Omega^\eps} \left( u^\eps - u_0 \right) \, \varphi^\eps \, dx_1 dx_2 \to 0 \quad \textrm{ and } \int_{\Omega^\eps} u_0 \, \left( \varphi^\eps - \phi \right) \, dx_1 dx_2 \to 0, \hbox{ as } \eps \to 0.
\end{gathered}
$$

Hence, since 
$$
\int_{\Omega^\eps} u_0(x_1) \, \phi(x_1) \, dx_1 dx_2 = \int_0^1 u_0(x_1) \, \phi(x_1) \, \left( G_\eps(x_1) + b(x_1) \right) \, dx_1,
$$
we get (\ref{INT3}) from (\ref{LGE}). 

\item Fourth integrand: we claim that
\begin{equation} \label{INT4}
\int_{\Omega^\eps} f^\eps \, \varphi^\eps \, dx_1 dx_2 \to \int_0^1  \hat f(x_1) \, \phi(x_1) \, dx_1 \textrm{ as } \eps \to 0.
\end{equation}
For this, let $\hat f \in L^2(0,1)$ be the function defined in (\ref{FHF}). 
Since
$$
\int_{\Omega^\eps} f^\eps \, \varphi^\eps \, dx_1 dx_2 = 
\int_{\Omega^\eps} f^\eps \, \left( \varphi^\eps - \phi \right) \, dx_1 dx_2
+ \int_{\Omega^\eps} f^\eps \, \phi \, dx_1 dx_2
$$ 
and 
$$
\int_{\Omega^\eps} f^\eps \, \phi \, dx_1 dx_2 = 
\int_0^1 \left( \int_{-b(x_1)}^{G_\eps(x_1)} f^\eps(x_1,x_2) \, dx_2 \right) \phi(x_1) \, dx_1
= \int_0^1 \hat f^\eps(x_1) \, \phi(x_1) \, dx_1,
$$
we get (\ref{INT4}) from (\ref{ESTF}), (\ref{LIMITF}) and (\ref{TFCONV}). 
\end{itemize}

Therefore, from (\ref{INT1}), (\ref{INT2}), (\ref{INT3}) and (\ref{INT4}) we obtain the following limit variational formulation 
\begin{equation} \label{limitP}
\int_0^1 \left\{ \left( G_0(x_1) + b(x_1) \right) \, u'_0(x_1) \, \phi'(x_1) + p(x_1) \, u_0(x_1) \, \phi(x_1) \right\} dx_1  
= \int_0^1 \hat f(x_1) \, \phi(x_1) \, dx_1,
\end{equation}
for all $\phi \in H^1(0,1)$. Since this problem is well posed, we obtain that the whole sequence $\{ u^\eps \}_{\eps > 0}$ is convergent and converges to the unique solution $u_0$ of (\ref{limitP}).

\par\bigskip\noindent {\bf  (e). Limit of $u^\epsilon$ in $H^1(\Omega_0)$}. 

Finally, let us show the strong convergence $u^\epsilon \to u_0$ in $H^1(\Omega_0)$.
We use that the norm is lower semicontinuous with respect to the weak convergence, that is, $\| u_0\|_{H^1(\Omega_0)} \leq \liminf_\epsilon \| u^\epsilon\|_{H^1(\Omega_0)}$. 
Indeed, since $u^\epsilon \rightharpoonup u_0 \quad w-H^1(\Omega_0)$ by \eqref{WC0}, we have for $\epsilon <<1$ that
\begin{eqnarray} \label{SEILA}
\int_{\Omega_0} {u'_0(x_1)}^2 \, dx_1 dx_2 & \leq & \liminf_\epsilon \int_{\Omega_0} | \nabla u^\epsilon |^2 \, dx_1 dx_2 \leq \limsup_\epsilon \int_{\Omega_0} | \nabla u^\epsilon |^2 \, dx_1 dx_2 \nonumber \\
& \leq & \limsup_\epsilon \int_{\Omega^\epsilon} | \nabla u^\epsilon |^2 \, dx_1 dx_2 \leq \limsup_\epsilon \left\{ \int_{\Omega^\epsilon} f^\epsilon \, u^\epsilon \, dx_1 dx_2 - \int_{\Omega^\epsilon} {u^\epsilon}^2 \, dx_1 dx_2 \right\} \nonumber  \\
& \leq & \limsup_\epsilon \Big\{ \int_{\Omega^\epsilon} f^\epsilon \left( u^\epsilon - u_0 \right) \, dx_1 dx_2 + \int_{\Omega^\epsilon} f^\epsilon \, u_0 \, dx_1 dx_2 \nonumber \\
& &  - \int_{\Omega^\epsilon} u_0^2 \, dx_1 dx_2 - \int_{\Omega^\epsilon} \left( {u^\epsilon}^2 - {u_0}^2 \right) \, dx_1 dx_2 \Big\} \\
& \leq & \int_0^1 \left\{ \hat f (x_1) \, u_0(x_1) - p(x) \, u_0(x_1)^2 \right\} dx_1 
= \int_0^1 \left( G_0(x_1) + b(x_1) \right) \, u'_0(x_1)^2 \, dx_1. \nonumber 
\end{eqnarray}
Here we have used \eqref{L2CONV}, \eqref{LIMITF} and \eqref{limitP}. 
Consequently, since 
$$
\int_{\Omega_0} {u'_0(x_1)}^2 \, dx_1 dx_2 = \int_0^1 \left( G_0(x_1) + b(x_1) \right) \, u'_0(x_1)^2 \, dx_1,
$$ 
we get $u^\epsilon \to u_0$ in $H^1(\Omega_0)$.
We conclude the proof of Theorem \ref{main}. 

\end{proof}

\section{Comb-like thin domains}
\label{comb}

We consider now, Type II thin domains as described in Section \ref{basics} and provide a proof of Theorem  \ref{main2}.

\begin{proof}

We will proceed as in the previous section to show this result. We will choose appropriate test functions 
to pass to the limit in the variational formulation of problem (\ref{P-typeII}) that we rewrite it here as:
find $u^\epsilon \in H^1(\Omega^\epsilon)$ such that 
\begin{eqnarray} \label{VFP3}
& & \int_{\Omega^\epsilon_+} \Big\{ \frac{\partial u^\epsilon}{\partial x_1} \frac{\partial \varphi}{\partial x_1} 
+ \frac{1}{\epsilon^2} \frac{\partial u^\epsilon}{\partial x_2} \frac{\partial \varphi}{\partial x_2} \Big\} dx_1 dx_2
+ \int_{\Omega_-} \Big\{ \frac{\partial u^\epsilon}{\partial x_1} \frac{\partial \varphi}{\partial x_1} 
+ \frac{1}{\epsilon^2} \frac{\partial u^\epsilon}{\partial x_2} \frac{\partial \varphi}{\partial x_2} \Big\} dx_1 dx_2
\nonumber \\
& & \quad \quad + \int_{\Omega^\epsilon} u^\epsilon \varphi \, dx_1 dx_2 = \int_{\Omega^\epsilon} f^\epsilon \varphi \, dx_1 dx_2, 
\, \forall \varphi \in H^1(\Omega^\epsilon).
\end{eqnarray}
Again, as in the previous case, taking $\varphi = u^\epsilon$ in  expression (\ref{VFP3}) and using that $\|f^\eps\|_{L^2(\Omega^\eps)}\leq C$,  we easily obtain the a priori bounds

\begin{equation} \label{EST0-II}
\begin{gathered}
\| u^\epsilon \|_{L^2(\Omega^\epsilon)}, \Big\| \frac{\partial u^\epsilon}{\partial x_1} \Big\|_{L^2(\Omega^\epsilon)}
\textrm{ and } \frac{1}{\epsilon} \Big\| \frac{\partial u^\epsilon}{\partial x_2} \Big\|_{L^2(\Omega^\epsilon)} 
\le C.
\end{gathered}
\end{equation}
In particular, we have 
$$
\Big\| \frac{\partial u^\epsilon}{\partial x_2} \Big\|_{L^2(\Omega^\epsilon)} \leq \eps \, C \to 0 \textrm{ as } \eps \to 0.
$$

We extract a subsequence of $\{ u^\eps|_{\Omega_-} \} \subset H^1(\Omega_-)$, denoted again by  $u^\epsilon$, such that
\begin{equation} \label{WCG}
\begin{gathered}
u^\epsilon \rightharpoonup u_0 \quad w-H^1(\Omega_-) \\
u^\epsilon \rightarrow u_0 \quad s-H^s(\Omega_-) \textrm{ for all } s \in [0, 1) \textrm{ and }\\
\frac{\partial u^\epsilon}{\partial x_2} \rightarrow 0 \quad s-L^2(\Omega_-) \\
\end{gathered}
\end{equation}
as $\epsilon \to 0$ for some  $u_0 \in H^1(\Omega_-)$.

As in (\ref{U0}), it follows from (\ref{WCG}) that $ u_0(x_1,x_2)$ does not depend on the variable $x_2$ and belongs to $H^1(0,1)$.  
Indeed, we can show that
$$
\frac{\partial  u_0}{\partial x_2}(x_1,x_2) = 0 \textrm{ a.e. } \Omega_-.
$$

\par\bigskip\noindent {\bf  (a). Limit of $u^\eps$ in $L^2(\Omega^\eps)$}. 

First, we obtain the $L^2$-convergence of $u^\eps$ to $u_0$.
More precisely, we show
\begin{equation} \label{L2CONVG}
\| u^\eps - u_0 \|_{L^2(\Omega^\eps)} \to 0 \textrm{ as } \eps \to 0.
\end{equation}

For this, we assume without loss of generality that 
$$
\Omega^\eps_+ \subset \{ (x_1,x_2) \in \R^2 \, \, | \, \, x_1 \in (0,1), \quad 0 < x_2 < b(x_1) \}
$$
and we define by `symmetrization' the following function $\hat u^\eps$ in $\Omega^\eps_+$ by
\begin{equation} \label{UHAT}
\hat u^\eps(x_1, x_2) = 
\left\{
\begin{array}{ll}
u^\eps(x_1, - x_2), & (x_1, x_2) \in \Omega^\eps_+ \\
u^\eps(x_1,x_2), & (x_1, x_2) \in \Omega_-.
\end{array}
\right.
\end{equation}
Consequently, it follows from (\ref{WCG}) that 
$$
\| \hat u^\eps - u_0 \|_{L^2(\Omega^\eps)} \to 0 \textrm{ as } \eps \to 0,
$$
and from \eqref{EST0-II}, we have
\begin{equation} \label{EST0-2-II}
\begin{gathered}
\| \hat u^\epsilon \|_{L^2(\Omega^\epsilon)}, \Big\| \frac{\partial \hat u^\epsilon}{\partial x_1} \Big\|_{L^2(\Omega^\epsilon)}
\textrm{ and } \frac{1}{\epsilon} \Big\| \frac{\partial \hat u^\epsilon}{\partial x_2} \Big\|_{L^2(\Omega^\epsilon)} 
\le C.
\end{gathered}
\end{equation}

Let us denote by  $w^\eps = u^\eps - \hat u^\eps$ in $\Omega^\eps$.
It is easy to see that $w^\eps \equiv 0$ in $\Omega_-$ and $w^\eps$ satisfies
\begin{equation} \label{ESTW}
\begin{gathered}
\| w^\eps \|_{H^1(\Omega^\eps_+)} = \| u^\eps - \hat u^\eps \|_{H^1(\Omega^\eps_+)} \leq C_1, \\
\left\| \frac{\partial w^\eps}{\partial x_1} \right\|_{L^2(\Omega^\eps_+)}^2 
+ \frac{1}{\eps^2} \left\| \frac{\partial w^\eps}{\partial x_2} \right\|_{L^2(\Omega^\eps_+)}^2 
+ \left\| w^\eps \right\|_{L^2(\Omega^\eps_+)}^2 \leq C_2.
\end{gathered}
\end{equation}

Now let us show that $\| w^\eps \|_{L^2(\Omega^\eps)} \to 0$ as $\eps \to 0$, that is, $\| w^\eps \|_{L^2(\Omega^\eps_+)} \to 0$ as $\eps \to 0$.  
Suppose this is not true and assume that $\| w^\eps \|_{L^2(\Omega^\eps_+)}^2 \geq c_0 > 0$ at least for a subsequence $\eps \to 0$.
Then we have that
$$
J(w^\eps) = \frac{\displaystyle\left\| \frac{ \partial w^\eps}{\partial x_1} \right\|_{L^2(\Omega^\eps_+)}^2 
+ \frac{1}{\eps^2} \left\| \frac{\partial w^\eps}{\partial x_2} \right\|_{L^2(\Omega^\eps_+)}^2 
+ \left\| w^\eps \right\|_{L^2(\Omega^\eps_+)}^2}{\displaystyle\| w^\eps \|^2_{L^2(\Omega^\eps_+)}} \leq \frac{C_2}{c_0} = C.
$$
This implies that the first eigenvalue of the problem
\begin{equation} \label{EVP}
\left\{
\begin{gathered}
- \frac{\partial^2 v^\epsilon}{{\partial x_1}^2} - 
\frac{1}{\epsilon^2} \frac{\partial^2 v^\epsilon}{{\partial x_2}^2} + v^\epsilon = \lambda_\eps \, v^\eps
\quad \textrm{ in } \Omega^\epsilon_+ \\
\frac{\partial v^\epsilon}{\partial x_1} \nu_1^\epsilon + \frac{1}{\epsilon^2} \frac{\partial v^\epsilon}{\partial x_2}\nu_2^\epsilon = 0
\quad \textrm{ on } \partial \Omega^\epsilon_+ \backslash \Gamma \\
v^\eps(x_1,0) = 0 \quad \textrm{ on } \Gamma
\end{gathered}
\right.
\end{equation}
satisfies $\lambda_\eps (\Omega^\eps_+) \leq C$, since $J$ is the associated Raleigh quotient 
and $\Gamma \subset \partial \Omega^\eps_+$ is a nonempty open subset.

But observe that $\Omega^\eps_+=\cup_{n=1}^{N_\eps}\Omega^\eps_{n,+}$ 
where all $\Omega^\eps_{n,+}$ are disjoint and identical,
except for translations. Then, we can conclude $\lambda_\eps (\Omega^\eps_+)=\lambda_\eps (\Omega^\eps_{n,+})$ for all $n$.

Performing in $\Omega^\eps_{n,+}$ the change of variables that transforms it into the fixed domain $Q_0$,
that is, $(x_1, x_2) \to (x_1/\eps^\alpha-nL,x_2)$, we will have that $\lambda_\eps(\Omega^\eps_{n,+})$ is the first eigenvalue of the problem
\begin{equation} \label{EVPF}
\left\{
\begin{gathered}
- \frac{1}{\eps^{2\alpha}} \frac{\partial^2 v^\epsilon}{{\partial x_1}^2} - 
\frac{1}{\epsilon^2} \frac{\partial^2 v^\epsilon}{{\partial x_2}^2} + v^\epsilon = \lambda_\eps \, v^\eps
\quad \textrm{ in } Q_0 \\
\frac{1}{\eps^{2\alpha}} \frac{\partial v^\epsilon}{\partial x_1} \nu_1^\epsilon + \frac{1}{\epsilon^2} \frac{\partial v^\epsilon}{\partial x_2}\nu_2^\epsilon = 0
\quad \textrm{ on } \partial Q_0 \backslash \Gamma_0 \\
v^\eps(x_1,0) = 0 \quad \textrm{ on } \Gamma_0
\end{gathered}
\right.
\end{equation}
and therefore,  
$$
\lambda_\eps(\Omega^\eps_+) = \min \left\{ \frac{ \frac{1}{\eps^{2 \alpha}} \int_{Q_0} \left| \frac{\partial^2 v^\epsilon}{{\partial x_1}^2} \right|^2 \, dx_1 dx_2 + \frac{1}{\epsilon^2} \int_{Q_0} \left| \frac{\partial^2 v^\epsilon}{{\partial x_2}^2} \right|^2 \, dx_1 dx_2 }{\int_{Q_0} \left| v^\eps \right|^2 \, dx_1 dx_2} \; \Big| \; v^\eps \in H^1(Q_0), \, v^\eps |_{\Gamma_0} = 0 \right\} \leq C.
$$

But this is impossible since $\lambda_\eps(\Omega^\eps_+) \geq \frac{1}{\eps^2} e_1(Q_0)$ 
for $\alpha >1$ and $\eps \in (0,1)$ where
$$
e_1(Q_0) = \min \left\{ \frac{ \int_{Q_0} \left( \left| \frac{\partial^2 v^\epsilon}{{\partial x_1}^2} \right|^2 + \left| \frac{\partial^2 v^\epsilon}{{\partial x_2}^2} \right|^2 \right) \, dx_1 dx_2 }{\int_{Q_0} \left| v^\eps \right|^2 \, dx_1 dx_2} \; \Big| \; v^\eps \in H^1(Q_0), \, v^\eps |_{\Gamma_0} = 0 \right\} 
$$
is the first eigenvalue of the Laplace operator in $Q_0$ with homogeneous Dirichlet boundary condition in $\Gamma_0$ and Neumann everywhere else. This eigenvalue is strictly positive by hypothesis (HQ). 
Thus we obtain (\ref{L2CONVG}).

\par\bigskip\noindent {\bf  (b). Test functions}. 

The test functions we are going to construct to pass the limit in the variational formulation (\ref{VFP3}) are
very similar to the ones we constructed in Type I thin domains. 
Take $\phi \in H^1(0,1)$, $\eps > 0$ 
and define the following functions in $H^1(\Omega^\eps)$:
\begin{equation} \label{TESTFG}
\begin{gathered}
\varphi^\eps(x_1,x_2) = \left\{
\begin{array}{ll}
X^\eps_n(x_1,x_2), & (x_1,x_2) \in \Omega^\eps_+\cap Q^\eps_n,\quad n=1,2,\ldots, N^\eps = \frac{1}{L \eps^\alpha} \\
\phi(x_1), & (x_1,x_2) \in \Omega_- 
\end{array}
\right.
\end{gathered}
\end{equation}
where  $Q^\eps_n$ is the rectangle 
$$
Q^\eps_n = (n L \eps^\alpha, (n+1) L \eps^\alpha) \times (0,G) 
$$
and the function $X^\eps_n$ is the solution of the problem
\begin{equation} \label{AUXSOLG}
\left\{
\begin{gathered}
- \frac{\partial^2 X^\eps}{\partial x_1^2} - \frac{1}{\eps^2} \frac{\partial^2 X^\eps}{\partial x_2^2} 
= 0, \quad \textrm{ in } Q^\eps_n \\
\frac{\partial X^\eps}{\partial N^\eps}=0, \quad \textrm{ on }  \partial Q^\eps_n \backslash \Gamma_n^\eps  \\
X^\eps(x_1,x_2) = \phi(x_1),  \quad \textrm{ on } \Gamma_n^\eps
\end{gathered}
\right. 
\end{equation}
where $\Gamma_n^\eps$ is the base of the rectangle, that is, 
$$
\Gamma_n^\eps = (n L \eps^\alpha, (n+1) L \eps^\alpha) \cap \partial Q_0.
$$

As we showed in the previous section, we have using Lemma \ref{basic-lemma}, 
\begin{eqnarray} \label{ESTXG2-a}
\left\| \frac{\partial^2 \varphi^\eps}{\partial x_1^2} \right\|^2_{L^2(\Omega^\eps_+)}
+ \frac{1}{\eps^2} \left\| \frac{\partial^2 \varphi^\eps}{\partial x_2^2} \right\|^2_{L^2(\Omega^\eps_+)}
 \leq  C \, \eps^{\alpha -1} \left\| \phi' \right\|^2_{L^2(0,1)}
\end{eqnarray}   
which implies that 
\begin{eqnarray} \label{ESTXG2}
\left\|  \varphi^\eps \right\|^2_{L^2(\Omega^\eps_+)}+\left\| \frac{\partial^2 \varphi^\eps}{\partial x_1^2} \right\|^2_{L^2(\Omega^\eps_+)}
+ \frac{1}{\eps^2} \left\| \frac{\partial^2 \varphi^\eps}{\partial x_2^2} \right\|^2_{L^2(\Omega^\eps_+)}
 \leq  C.
 \end{eqnarray}

Moreover, we can show that
\begin{equation} \label{TFCONVG}
\| \varphi^\eps - \phi \|_{L^2(\Omega^\eps)} \to 0 \textrm{ as } \eps \to 0.
\end{equation}
We can argue as in (\ref{L2CONVG}). If it were not true, then there will exist a $c_0>0$ and a sequence
(that we still denote it by $\eps$) such that $\| \varphi^\eps - \phi \|_{L^2(\Omega^\eps)}\geq c_0$. But then, if we define $w^\eps=\varphi^\eps - \phi$, we will have that 
$$
J(w^\eps) = \frac{\displaystyle\left\| \frac{ \partial w^\eps}{\partial x_1} \right\|_{L^2(\Omega^\eps_+)}^2 
+ \frac{1}{\eps^2} \left\| \frac{\partial w^\eps}{\partial x_2} \right\|_{L^2(\Omega^\eps_+)}^2 
+ \left\| w^\eps \right\|_{L^2(\Omega^\eps_+)}^2}{\displaystyle\| w^\eps \|^2_{L^2(\Omega^\eps_+)}} \leq \frac{C}{c_0} = \tilde C
$$
but with the same steps as we did in (a) this will contradict the fact that $e_1(Q_0)>0$.

\par\bigskip\noindent {\bf  (c). Pass to the limit}. 

Now we can pass to the limit in the variational formulation (\ref{VFP3}).
First, we note that the convergence of
\begin{equation} \label{INTG4}
\begin{gathered}
\int_{\Omega^\epsilon_+} \Big\{ \frac{\partial u^\epsilon}{\partial x_1} \frac{\partial \varphi^\eps}{\partial x_1} 
+ \frac{1}{\epsilon^2} \frac{\partial u^\epsilon}{\partial x_2} \frac{\partial \varphi^\eps}{\partial x_2} \Big\} dx_1 dx_2 \to 0,\hbox{ as } \eps\to 0
\end{gathered}
\end{equation}
follows from (\ref{ESTXG2-a}) and can be obtained as in (\ref{INT1}). 

Also, from  \eqref{WCG} and since $\varphi^\eps\equiv \phi$ in $\Omega_-$, we easily get
\begin{equation} \label{INTG3}
\int_{\Omega_-} \Big\{ \frac{\partial u^\epsilon}{\partial x_1} \frac{\partial \varphi^\eps}{\partial x_1} 
+ \frac{1}{\epsilon^2} \frac{\partial u^\epsilon}{\partial x_2} \frac{\partial \varphi^\eps}{\partial x_2} \Big\} dx_1 dx_2
\to \int_0^1 b(x_1) \, u'_0(x_1) \,  \phi'(x_1) \, dx_1 \textrm{ as } \eps \to 0.
\end{equation}

Let us consider now the following technical result. 
\begin{lemma}\label{lemma-technical}
We have
$$
\int_{S^\eps(x_1)}  dx_2\to q(x_1)\equiv \frac{|Q_0|}{L} + b(x_1) \qquad w^*-L^\infty(0,1).
$$
\end{lemma}
\begin{proof} If we denote by 
$\chi$ the characteristic function of the measurable open set $Q_0$, extended periodically with respect to the
first variable, we have by the Average Theorem that
\begin{eqnarray} \label{WCTS}
\int_{S^\eps(x_1)}  dx_2 & = & \int_0^G \chi(x_1/\eps^\alpha,x_2) \, dx_2  + \int_{-b(x_1)}^0 dx_2 \nonumber \\
& \rightharpoonup & \int_0^G \left( \frac{1}{L} \int_0^L  \chi(s, x_2) \, ds \right) dx_2 + b(x_1) 
\qquad w^*-L^\infty(0,1) \nonumber \\
& = & \frac{|Q_0|}{L} + b(x_1) \quad \forall x_1 \in (0,1).  \nonumber 
\end{eqnarray}
\end{proof}

Moreover, 
\begin{eqnarray*}
\int_{\Omega^\eps} u^\eps \, \varphi^\eps \, dx_1 dx_2 & = & 
\int_{\Omega^\eps} \left( u^\eps - u_0 \right) \, \varphi^\eps \, dx_1 dx_2
+ \int_{\Omega^\eps} u_0 \, \left( \varphi^\eps - \phi \right) \, dx_1 dx_2
+ \int_{\Omega^\eps} u_0 \, \phi \, dx_1 dx_2
\end{eqnarray*}
and the first two integrals go to 0 since $\|u^\eps-u^0\|_{L^2(\Omega^\eps)}\to 0$ and 
$\|\varphi^\eps-\phi\|_{L^2(\Omega^\eps)}\to 0$.  The last integral satisfies,
$$
\int_{\Omega^\eps} u_0(x_1) \, \phi(x_1) \, dx_1 dx_2 = \int_0^1 u_0(x_1) \, \phi(x_1) \, \left( \int_{S^\eps(x_1)}  dx_2 \right) \, dx_1 \to \int_0^1 q(x_1)u_0(x_1) \, \phi(x_1)dx_1
$$
where we have used Lemma \ref{lemma-technical}.

Finally, we have 
$$
\int_{\Omega^\eps} f^\eps \, \varphi^\eps \, dx_1 dx_2 = 
\int_{\Omega^\eps} f^\eps \, \left( \varphi^\eps - \phi \right) \, dx_1 dx_2
+ \int_{\Omega^\eps} f^\eps \, \phi \, dx_1 dx_2
$$ 
but the first integral goes to 0.  Moreover, with the hypothesis of the theorem, we get for the
second integral
$$
\int_{\Omega^\eps} f^\eps \, \phi \, dx_1 dx_2 = 
\int_0^1 \left( \int_{S^\eps(x_1)} f^\eps(x_1,x_2) \, dx_2 \right) \phi(x_1) \, dx_1\to  \int_0^1 \hat f (x_1) \, \phi(x_1) \, dx_1. 
$$

Therefore, we obtain from the estimates above that 
\begin{equation} \label{limitPG}
\int_0^1 \left\{ b(x_1) \, u'_0(x_1) \, \phi'(x_1) + q(x_1) \, u_0(x_1) \, \phi(x_1) \right\} dx_1  
= \int_0^1 \hat f(x_1) \, \phi(x_1) \, dx_1 \quad \forall \phi \in H^1(0,1).
\end{equation}
Since this problem has a unique solution, then we obtain that the sequence $\{ u^\eps \}_{\eps > 0}$ is convergent and converges to the unique solution $u_0$ of (\ref{limitPG}).

 Moreover, arguing as \eqref{SEILA} at the last section, we obtain the strong convergence $u^\epsilon \to u_0$ in $H^1(\Omega_-)$ by \eqref{L2CONVG}, \eqref{LIMITF} and \eqref{limitPG} concluding the proof.

\end{proof}


\end{document}